\documentclass[11pt]{article}
\usepackage{amsfonts,amsmath,latexsym,color,epsfig}
\newtheorem{theorem}{Theorem}
\newtheorem{lemma}[theorem]{Lemma}

\newenvironment{proof}
     {\medskip\noindent{\bf Proof:}\hspace{1mm}}
      {\hfill$\Box$\medskip}

\def\qed{\ifvmode\mbox{ }\else\unskip\fi\hskip 1em plus 10fill$\Box$}
\long\def\ignore#1{}

\def\rc{\advance\leftskip by 0pt plus 40em\rightskip=\leftskip
\parfillskip=0pt
\spaceskip=.3333em \xspaceskip=.5em \pretolerance=9999
\tolerance=9999 \hyphenpenalty=9999 \exhyphenpenalty=9999}

\title{The Range of a Random Walk on a Comb}

\author{J\'anos Pach\thanks{EPFL, Lausanne and R\'enyi Institute,
    Budapest. Supported by NSF Grant CCF-08-30272, by OTKA under EUROGIGA
    projects GraDR and ComPoSe 10-EuroGIGA-OP-003, and by Swiss National
    Science Foundation Grants 200021-137574 and 200020-144531. Email: {\tt pach@cims.nyu.edu}}
  \and G\'abor Tardos\thanks{R\'enyi Institute, Budapest. Supported by an NSERC
    grant and the OTKA grant NN102029. Email: {\tt
      tardos@renyi.hu}}}
\date{}
\begin{document}
\maketitle
\begin{abstract}
The graph obtained from the integer grid
$\mathbb{Z}\times\mathbb{Z}$ by the removal of all horizontal
edges that do not belong to the $x$-axis is called a {\em comb}.
In a random walk on a graph, whenever a walker is at a vertex $v$,
in the next step it will visit one of the neighbors of $v$, each
with probability $1/d(v)$, where $d(v)$ denotes the degree of $v$.
We answer a question of Cs\'aki, Cs\"org\H o, F\"oldes,
R\'ev\'esz, and Tusn\'ady by showing that the expected number of
vertices visited by a random walk on the comb after $n$ steps is
$\left(\frac1{2\sqrt{2\pi}}+o(1)\right)\sqrt{n}\log n.$ This
contradicts a claim of Weiss and Havlin.
\end{abstract}

\section{Introduction}
The theory of finite Markov chains or, in graph-theoretic
language, the theory of random walks on graphs is a classical
topic in probability theory. It has many applications from flows
in networks, through statistical physics to complexity theory in
computer science (see Lov\'asz~\cite{Lo96} and Woess~\cite{Wo00}).

\smallskip
To obtain a {\em random walk} in a locally finite graph $G$ with
vertex set $V(G)$, we start at any vertex $v\in V(G)$ and in the
next step we move to one of its neighbors, independently of all
previous events, with probability $1/d(v)$. Here, $d(v)$ denotes
the number of edges in $G$ incident to $v$. Every neighbor of $v$
is equally likely to come next. Perhaps the simplest example is a
random walk on the $d$-dimensional integer grid $\mathbb{Z}^d$,
studied by P\'olya~\cite{Po21}. He proved that for $d=1$ and $2$,
with probability $1$, a random walk will return to its starting
point infinitely often, while for $d\ge 3$ only a finite number of
times.

\smallskip
The $2$-dimensional {\em comb} $\mathbb{C}^2$ is a spanning tree
of the integer grid $\mathbb Z^2$ obtained by removing all of its
``horizontal'' edges (that is, edges parallel to the $x$-axis)
that do not belong to the $x$-axis. In this graph, all vertices
(or ``sites'') $(x,y)$ have degree $2$, except for the vertices of
the form $(x,0)$, which have degree $4$. For any nonnegative
integer $n$, let $W_n=(X_n, Y_n)$ be the random variable denoting
the position of the walker after $n$ steps. We assume the walk
starts at the origin, so $W_0=(0,0)$. The study of random walks on
the comb was initiated by Weiss and Havlin~\cite{WeH86}, as a
model of ``anomalous diffusion on fractal structures.''  These
investigations were later extended to combs of higher dimensions
by Gerl~\cite{Ge86} and Cassi and Regina~\cite{CaR92}. Krishnapur
and Peres~\cite{KrP04} proved that, on the $2$-dimensional comb,
with probability $1$, two independent walkers meet only a finite
number of times. This is a rather surprising phenomenon, in view
of the fact that the random walk is {\em recurrent},
that is, a single random walker visits each site an infinite
number of times with probability 1. Some insight was provided by
Bertacchi and Zucca~\cite{BeZ03} and by Bertacchi~\cite{Be06},
whose asymptotic estimates suggested that a walker spends most of
her time moving vertically along a ``tooth'' of the comb. Several
strong approximation and limit theorems for random walks on a comb
have been established by Cs\'aki, Cs\"org\H o, F\"oldes, and
R\'ev\'esz~\cite{CsCs09},\cite{CsCs11}.

\smallskip
Let $V_n$ denote the number of vertices (sites) visited during the
first $n$ steps of the random walk $W_n=(X_n,Y_n)$ on the
$2$-dimensional comb. According to the main result in
\cite{WeH86}, the expected value of $V_n$ is asymptotically
proportional to $n^{3/4}$, for large $n$. It is not hard to see
that almost surely the deviation of the horizontal projection
$X_n$ of the walk is roughly $n^{1/4}$, while the expected length
of the vertical projection is of order $n^{1/2}$. See, e.g.,
Bertacchi~\cite{Be06} (cp. Panny and Prodinger~\cite{PaP85}). This
suggests that the expected number of sites visited by the random
walk on $\mathbb{C}^2$ is around $n^{1/4}\cdot n^{1/2}=n^{3/4}$,
as was stated by Weiss and Havlin~\cite{WeH86}. The aim of this
note is to show that the truth is closer $n^{1/2}$, than to
$n^{3/4}$.

\smallskip
 All logarithms used in this paper are natural logarithms.
\medskip
\begin{theorem}\label{fo} The expected value of $V_n$, the number of
vertices visited during the first $n$ steps of a random walk on the
$2$-dimensional comb, satisfies
$$E[V_n]=\left(\frac1{2\sqrt{2\pi}}+o(1)\right)\sqrt{n}\log n.$$
\end{theorem}

\section{Elementary properties of a random walk on $\mathbb{Z}$}\label{line}
We collect some well-known and easy facts about 1-dimensional
random walks on $\mathbb{Z}$; all of them can be found, e.g., in
\cite{Fe68} or \cite{Re07}. For any pair of integers $n, i \ge 0$,
let $p_{n,i}$ denote the probability that starting at $0$, after
$n$ steps we end up at the vertex (integer) $i$. We have
$$p_{n,i}={n \choose \frac{n+i}2}\frac1{2^{n}},$$
where the value of the above binomial coefficient is considered
$0$, whenever $\frac{n+i}2$ is not an integer. It follows that
$$p_{n,i}\le {n \choose \lfloor{n/2}\rfloor}\frac1{2^{n}}
=\left(\sqrt{\frac2{\pi}}+o(1)\right)\frac1{\sqrt{n}},$$ as $n$
tends to infinity. Moreover, we have

\begin{equation}\label{prob}
p_{n,i}=\left(\sqrt{\frac2{\pi}}+o(1)\right)\frac1{\sqrt{n}},
\end{equation}
whenever $i/\sqrt{n}\rightarrow 0$ and $n+i$ is even.

\smallskip
Let $A_n$ stand for the number of times the random walk visits the
origin during the first $n$ steps and $B_n$ for the number of
sites visited during the the first $n$ steps. We have
\begin{equation}\label{A_n}
E[A_n]=\sum_{m=0}^np_{m,0}=\left(\sqrt{\frac2{\pi}}+o(1)\right)\sqrt
n,
\end{equation}
\begin{equation}\label{B_n}
E[B_n]=\left(2\sqrt{\frac2{\pi}}+o(1)\right)\sqrt n.
\end{equation}
See, e.g., \cite{Re07}, p. 253.
\smallskip
Finally, for any $j > 0$, let $r_j$ denote the probability that
starting at position $0$, the infinite random walk on $\mathbb{Z}$
reaches $j$ before it would return to $0$. We have
\begin{equation}\label{noreturn}
r_j=\frac1{2j}.
\end{equation}

\section{Proof of Theorem~\ref{fo}}
The vertices and edges of the comb $\mathbb{C}^2$ that belong to
the $x$-axis form the {\em backbone}. The connected components of
the graph obtained from $\mathbb{C}^2$ after the removal of the
backbone are called {\em teeth}.

\smallskip
We consider the projections $(X_i)$ and $(Y_i)$ of the
two-dimensional walk $(X_i,Y_i)$, separately. First, we reduce
these one-dimensional walks by getting rid of the steps when the
value does not change. In this way, horizontal steps contribute
only to the {\em reduced projection} $(X'_i)$ and vertical steps
contribute only to the {\em reduced projection} $(Y'_i)$. Note
that with probability $1$ both reduced walks are infinite and they
are distributed as the standard random walk on the line. Let us
consider the random walk on the comb up to (and including) the
$n$\/th {\em vertical} move. We call this walk $W$. Its reduced
projections are $(X'_i)_{i=0}^a$ and $(Y'_i)_{i=0}^n$, where the
random variable $a$ is the number of horizontal moves in $W$. Let
$c$ stand for the number of sites {\em on the backbone} reached by
$W$. It is easy to describe the asymptotic behavior of $E[a]$ and
$E[c]$.

\begin{lemma}\label{lem:ac}
$$E[a]=\left(\sqrt{\frac2{\pi}}+o(1)\right)n^{1/2};$$
$$E[c]=\Theta(n^{1/4}).$$
\end{lemma}
\begin{proof}
Let $d$ denote the number of times $W$ moves from a tooth to a
position on the backbone, including the starting position at the
origin, but not including the final position after the last move,
even if it satisfies this condition. Clearly, $d$ is distributed
as $A_{n-1}$. $W$ has $d$ chances to make horizontal moves. At
each chance, it makes precisely $i$ consecutive horizontal moves
with probability $2^{-i-1}$. Thus, for the expected value of $a$,
the number of horizontal moves in $W$, we have
$$E[a]=E[d\sum_{i=0}^{\infty}i2^{-i-1}]=E[d]=E[A_{n-1}]
=\left(\sqrt{\frac2{\pi}}+o(1)\right)\sqrt n.$$ Here the last
equality follows by (\ref{A_n}).

\smallskip
Obviously, the conditional distribution of $c$, given $a$, is the
same as the distribution of $B_a$, so that by (\ref{B_n}) we
obtain
$$E[c|a]=\left(2\sqrt{\frac2{\pi}}+o(1)\right)\sqrt a.$$
Hence, we have
\begin{equation}\label{Ec}
E[c]=\left(2\sqrt{\frac2{\pi}}+o(1)\right)E[\sqrt a]\le
\left(2\sqrt{\frac2{\pi}}+o(1)\right)\sqrt{E[a]}=\left(\frac{2^{7/4}}{\pi^{3/4}}+o(1)\right)n^{1/4}.
\end{equation}
On the other hand, the limiting distribution of $d/\sqrt n$ (i.e.,
that of $A_{n-1}/\sqrt n$) is equal to the distribution of the
absolute value of a random variable with standard normal
distribution (see, e.g., \cite{Re07}, p. 506). Therefore, we have
$P(d/\sqrt{n}\ge 1/2)\ge 1/2+o(1)$. Using the fact that, given
$d$, the inequality $a\ge d$ holds with probability exactly $1/2$,
we obtain that $E[\sqrt a]\ge(1/4+o(1))\sqrt{\frac12\sqrt
n}=\Theta(n^{1/4})$. In view of (\ref{Ec}), this implies the
second part of the lemma.
\end{proof}

\smallskip
We prove Theorem~\ref{fo} by estimating the expected number of
sites $V'_n$ reached by $W$. As $W$ makes $n+a$ moves, it reaches
the $V_n$ sites visited in the first $n$ steps and potentially at
most $a$ further sites. According to the first part of
Lemma~\ref{lem:ac}, this potential increase is so small that it
does not change the asymptotic behavior of the expected value.

\smallskip
First, we establish the upper bound. We classify a site $(i,j)$ as
{\em close} if $|j|<n^{1/4}$, {\em far} if
$|j|>2n^{1/2}\log^{1/2}n$, and {\em intermediate} if
$n^{1/4}\le|j|\le2n^{1/2}\log^{1/2}n$, and estimate the expected
number of sites reached in each class, separately.

The number of close sites reached is less than $2n^{1/4}+1$ times
the number $c$ of sites reached on the backbone. Thus, by
Lemma~\ref{lem:ac}, the expected number of close sites reached is
$O(\sqrt n)$. The number of far sites reached can easily be
bounded by the number of steps $0\le i\le n$ with
$|Y'_i|>2n^{1/2}\log^{1/2}n$. For each $i\le n$, the Chernoff bound
gives $P(|Y'_i|>2n^{1/2}\log^{1/2}n)<1/n$. Therefore, the expected
number of far sites reached is at most $1$.

Let us call the tooth of the comb containing the site where $W$
ends the {\em final tooth}. In case $Y'_n=0$, there is no final
tooth. The total number of intermediate sites on the final tooth
is less than $n^{1/2}\log^{1/2}n$, so this also bounds the
expected number of intermediate sites reached on the final tooth.

Finally, we consider the number $w$ of intermediate sites reached
outside the final tooth. When such a site at distance $j$ is
reached, it must be reached through a vertical move. If the $i$-th
vertical move reaches it {\em for the last time} in $W$, then we
have $|Y'_i|=j$ and the part of the walk after the $i$-th vertical
move (starting from this site) must reach the backbone before it
comes back to the same site. The probability for $|Y'_i|=j$ is
$2p_{i,j}$. Assuming that this happens, according to
(\ref{noreturn}), the probability that the infinite random walk
reaches the backbone before it returns to the same site is
$r_j=1/(2j)$. Hence, the total probability is at most $p_{i,j}/j$.
We bound $E[w]$ by summing $p_{i,j}/j$ over all $1\le i\le n$ and
$n^{1/4}\le j\le2n^{1/2}\log^{1/2}n$. Using the fact that
$p_{i,j}=0$ if $i+j$ is odd and, by (\ref{prob}),
$p_{i,j}\le(\sqrt{2/{\pi}}+o(1))/\sqrt{i}$, it follows that
\begin{equation}\label{upper}
E[w]\le\sum_j\sum_ip_{i,j}/j\le\left(\sqrt{\frac2{\pi}}+o(1)\right)\sum_{i,j\atop i+j\mbox{ \scriptsize even}}\frac1{j\sqrt i}\le\left(\sqrt{\frac2{\pi}}+o(1)\right)\frac14\sqrt n\log n.
\end{equation}
Summing over all the sites reached by $W$, we obtain the upper
bound in the theorem:
$$E[V_n]\le E[V'_n]\le\left(\frac1{2\sqrt{2\pi}}+o(1)\right)\sqrt{n}\log n.$$

\smallskip
Before turning to the lower bound, we introduce the symbol $u_j$
to denote the probability that $W$ reaches the site $(0,j)$. We
need the following lemma.

\begin{lemma}\label{lem:u}
$$u_j=O\left(\frac{n^{1/4}\log n}{|j|+1}\right).$$
\end{lemma}
\begin{proof}
Let $t_i$ stand for the probability that $W$ reaches the site
$(i,0)$. Clearly, we have $E[c]=\sum_{i=-\infty}^{+\infty}t_i$.
Let $W'$ denote the random walk on the comb starting at the origin
and ending with the $2n$-th vertical move. For the number
$V'_{2n}$ of the sites $W'$ reached, we have
\begin{equation}\label{reached}
E[V'_{2n}]=O(\sqrt n\log n), \end{equation} by the upper bound
(\ref{upper}) we have just proved. For every $i$ and every $j$,
the probability that $W'$ reaches the site $(i,j)$ is at least
$t_iu_j$. Indeed, with probability $t_i$ the walk reaches $(i,0)$
before the $n$-th vertical move, and after that it reaches $(i,j)$
within $n$ further vertical steps with probability $u_j$. Note
that if $(i,j)$ is reached, then at least $|j|+1$ sites of the
form $(i,j')$ are reached (along the same tooth). Therefore, we
obtain
$$E[V'_{2n}]\ge\sum_i(|j|+1)t_iu_j=(|j|+1)u_jE[c].$$ From here,
applying our upper bound (\ref{reached}) and Lemma~\ref{lem:ac},
the result follows. \end{proof}

\medskip

Now we are ready to prove the lower bound in Theorem~\ref{fo}.

Let $Z_i$ stand for the site reached by the $i$-th vertical move
of the random walk $W$. Let $q_{i,j}$ stand for the probability
that $|Y'_i|=j$ and $Z_i$ is reached by the $i$-th vertical move
{\em for the last time} in $W$, $0<i\le n$. Each site that does
not belong to the backbone and is reached by $W$, is reached for
the last time by a single uniquely determined vertical move. Thus,
we have
$$E[V'_n]=E[c]+\sum_{0<j,0<i\le n}q_{i,j}.$$

Suppose that $Z_i$ is not on the backbone. If after the $i$-th
vertical move the random walk $W$ returns to the backbone before
it would revisit $Z_i$, and after returning to the backbone it
still does not visit $Z_i$ during the following $n$ vertical
moves, then $Z_i$ was reached in $W$ for the last time by the
$i$-th vertical move. Therefore, using the notation in
Section~\ref{line}, we have
\begin{equation}\label{q}
q_{i,j}\ge 2p_{i,j}r_{j}(1-u_j)\;\;\; {\rm for}\;\;\: 0<i\le n,\;
j>0.
\end{equation} The lower bound for $E[V'_n]\ge\sum q_{i,j}$ can be
obtained by evaluating the terms $q_{i,j}$. Let us consider only
those terms $q_{i,j}$ for which $n^{1/4}\log^2n<j<n^{1/2}/\log
n,\;$ $j^2\log n\le i\le n$, and $i+j$ is even. For these values,
by Lemma~\ref{lem:u}, we have $u_j=o(1)$. Since
$j/\sqrt{i}\rightarrow 0$, (\ref{prob}) yields that
$p_{i,j}=(\sqrt{2/\pi}+o(1))/\sqrt{i}$. Combining this with
(\ref{noreturn}), inequality (\ref{q}) gives
$$q_{i,j}\ge\left(\sqrt{\frac{2}{\pi}}+o(1)\right)\frac1{\sqrt{i}j}.$$
Thus, for a fixed $j$, we get
$$\sum_{j^2\log n\le i\le n,\;i\equiv j} q_{i,j}\ge
\left(\sqrt{\frac{2}{\pi}}+o(1)\right)\frac{\sqrt{n}}{j}.$$
Summing over all $j$, $n^{1/4}\log^2n<j<n^{1/2}/\log n$, we obtain
$$E[V'_n]\ge\left(\sqrt{\frac{2}{\pi}}+o(1)\right)\sqrt{n}\sum_j\frac1j
\ge\left(\frac{1}{2\sqrt{2\pi}}+o(1)\right)\sqrt n\log n,$$ as
claimed. Note that our estimates for the expectation of $V'_n$
carry over to the expectation of $V_n$, as $|V'_n-V_n|\le a$ and
$E[a]=O(\sqrt n)$. This completes the proof of Theorem~\ref{fo}.
$\Box\Box$

\bigskip
\noindent{\bf Acknowledgment.} We are indebted to G. Tusn\'ady for
calling our attention to the problem addressed in this note, and
to E. Cs\'aki, M. Cs\"org\H o, A. F\"oldes, and P. R\'ev\'esz for 
their valuable remarks.

\end{document}